\documentclass[11pt,a4paper]{article}
\usepackage{amsmath, amsthm, amscd, amsfonts, amssymb, graphicx, color}
\usepackage[colorlinks=true, linkcolor = blue, citecolor = blue,
]{hyperref}

\usepackage{breqn}
\usepackage{csquotes}
\newtheorem{theorem}{Theorem}[section]
\newtheorem{lemma}[theorem]{Lemma}
\newtheorem{proposition}[theorem]{Proposition}
\newtheorem{corollary}[theorem]{Corollary}

\theoremstyle{definition}
\newtheorem{definition}[theorem]{Definition}

\theoremstyle{conclusion}
\newtheorem*{conclusion}{Conclusion}

\newtheorem{example}[theorem]{Example}

\theoremstyle{remark}

\numberwithin{equation}{section}
\usepackage{authblk}

\usepackage{textcomp}
\usepackage{eurosym}
\usepackage{ragged2e}

\usepackage[backend=biber,bibencoding=utf8, citestyle=numeric]{biblatex} 
\begin{document}
	\title{THEORY OF HYPERSURFACES OF A FINSLER SPACE WITH THE GENERALIZED SQUARE METRIC}
	\author[$\dagger$]{Sonia Rani}
	\author[$\ddagger$]{Vinod Kumar}
	\author[$\S$]{Mohammad Rafee \thanks{Corresponding authour: Mohammad Rafee; Email Id: mohd\textunderscore rafee60@yahoo.com}}
	
	\affil[$\dagger$,$\ddagger$]{Department of Mathematics, School of Applied Sciences, Om Sterling Global University, Hisar, India. }
	\affil[$\S$]{Department of Mathematics, School of Basic, Applied and Bio-Sciences, RIMT University,  India.}
	\date{ }
	\maketitle

\begin{abstract}
	The emergence of generalized square metrics in Finsler geometry can be attributed to various classification  concerning $(\alpha,\beta)$-metrics. They have excellent geometric properties in Finsler geometry. Within the scope of this research paper, we have conducted an investigation into the generalized square metric denoted as $F(x,y)=\frac{[\alpha(x,y)+\beta(x,y)]^{n+1}}{[\alpha(x,y)]^n}$, focusing specifically on its application to the Finslerian hypersurface.  Furthermore, the classification and existence of first, second, and third kind of  hyperplanes of the Finsler manifold has been established.
	
	
\end{abstract}
\begin{center}
\section{Introduction}
\end{center}

\textcolor{white}{"}Let $M$ be an $n$-dimensional differential manifold. Define a Finsler metric on the  differentiable manifold $M$. This Finsler metric is known as Finsler fundamental fuction on the manifold $M$. Let us first define what is exactly mean by Finsler fundamental function.

\begin{definition}[Finsler metric]
	\justifying
	
	\label{difinition1.0.8}
	We say a function $F:TM\rightarrow R$ is a Finsler metric or Finsler fundamental function on the manifold $M$ if $F$ satisfy the following conditions:\\

	\begin{enumerate}
		\raggedright
		\item $F$ is $C^\infty$ away from zero vectors of the tangent spaces:\\
		That is, $F$ is smooth on $TM\backslash\{0\}=\left\{(x,y)\vert x\in M, y\in T_xM, y\neq0\right\}$. The smoothness property is desired so that we can apply differential calculus on the Finsler metric $F$.
		\item Positivity of function $F$: \\
		$F(x,y)\geq0$ for $x\in M$ and all $y\in  T_xM$.\\
		 This property ensures that the length of a tangent vector $y\in T_xM$ is either positive or zero. In terms of arc length this property also ensures that arc length defined by the integral ~\ref{eq:1.1}  is either positive or zero.
		
		\item Positive homogeneity of function $F$: \\
		$F(x,\lambda y)=\lambda F(x,y)$, $\forall$ $\lambda>0$; for $x\in M$ and all $y\in  T_xM$.\\
		 That is, $F$ is +ve 1-homogeneous of first degree in the directional argument $y$.
		 This property ensures that, length of a tangent vector $\lambda y \in T_xM$ is nothing but $\lambda \times F(x,y)\in R$.
		 
		 \begin{figure}[h!]
		 	\includegraphics[scale=0.35]{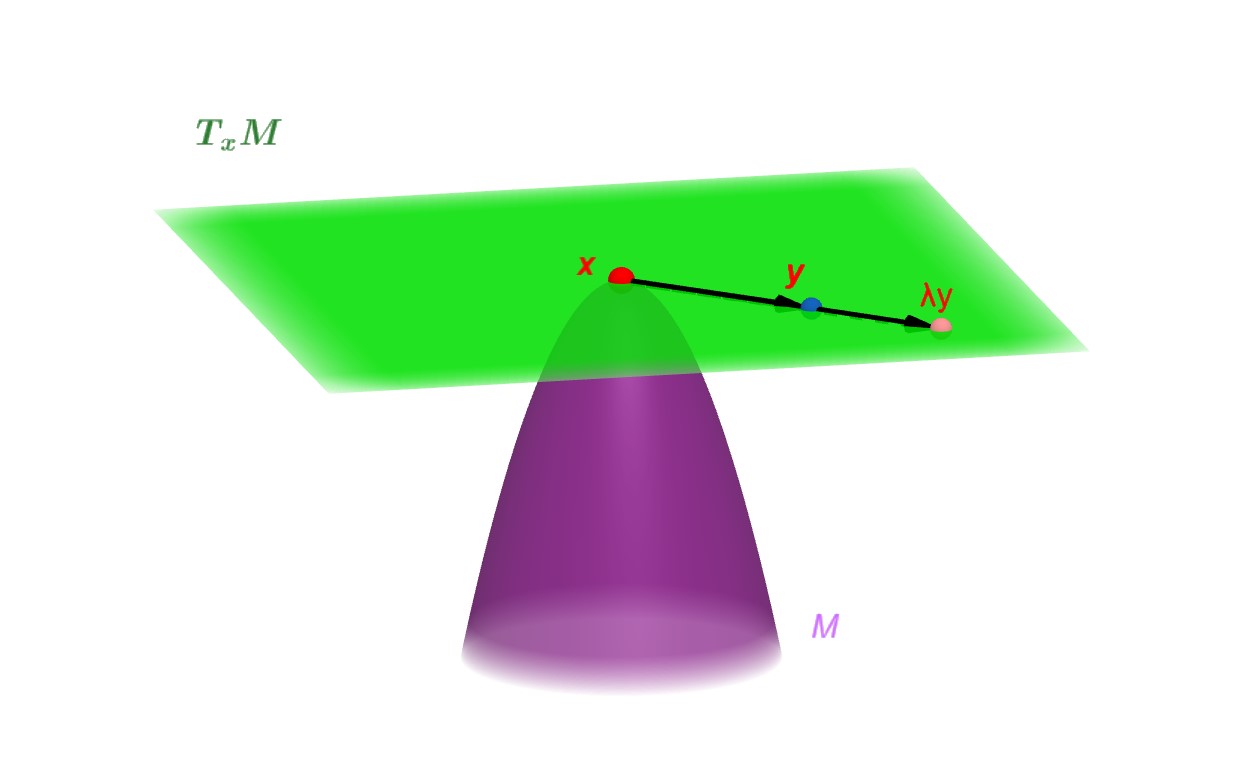}
		 	\caption{This figure shows the tangent vector $\lambda y\in T_xM$, $\lambda>0$, is  $\lambda$ times the tangent vector $y\in T_xM$.}
		 	\label{figure:1}
		 \end{figure}

		\item Strict convexity of the function $F$: \\ $F:TM\rightarrow R$ is strictly convex over the tangent bundle $TM$. 
	\end{enumerate}
It is important to note that strictly convex condition is equivalent to the  hessian matrix $[g_{ij}]$, where $i,j\in \left\{1,2,3,......dim(M)\right\}$, defined by $\frac{1}{2}\frac{\partial^2F^2}{\partial y^i\partial y^j}(x,y)=g_{ij}(x,y)$ is positive definite for any $(x,y)\in{TM}$. It is the convexity condition on the Finsler metric  $F$ that guaranties for  the arc length minimization, given by the following formula, of the admissible curves $\gamma:[a,b]\rightarrow M$ belonging to the set  $C^\infty[a,b]$\\

\begin{equation}
	\label{eq:1.1}
	s[\gamma(t)]=\int_{a}^{b}F(\gamma(t),\dot{\gamma}(t))dt
\end{equation}  
can be achieved. \\

\begin{figure}[h!]
	\includegraphics[scale=0.50]{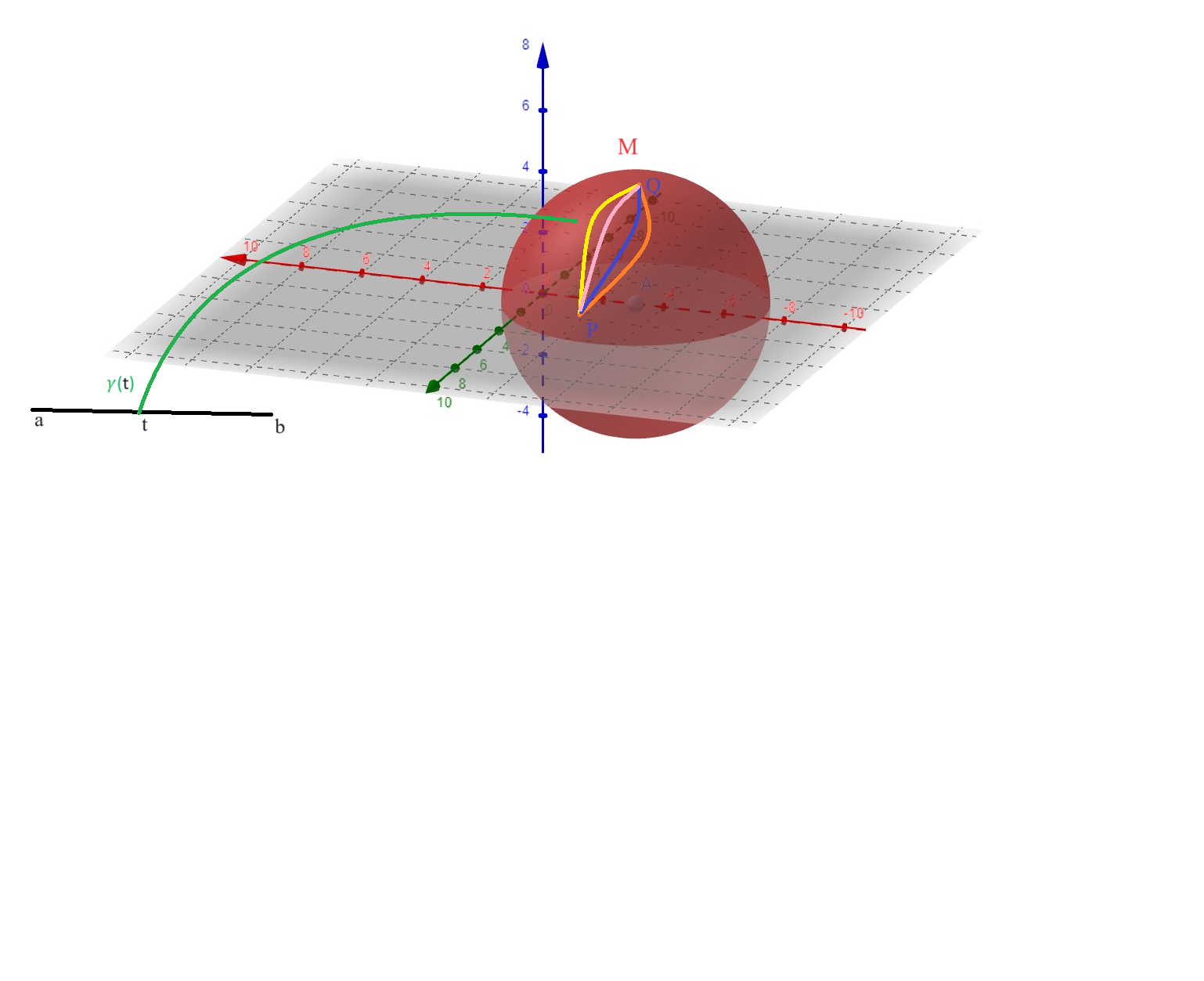}
	\caption{This figure shows the curves $\gamma:[a,b]\rightarrow M$ over the manifold $M$ such that $\gamma(a)=P$ and $\gamma(b)=Q$.}
	\label{figure:2}
\end{figure}

In other words, the convexity condition on the Finsler metric is a geometric requirement that makes sure that the length of a curve can be defined and that the arc length functional  $	s[\gamma(t)]$ is well-behaved, particularly for the purpose of minimizing or finding geodesics  with respect to the given Finsler metric $F$. Without convexity, the concept of length and the corresponding optimization problems will not make sense or may not have unique solutions.

\end{definition}
In the definition provided above, $F(x, y)$ signifies the magnitude of the vector $y$ within the tangent space $T_xM$ originating from an arbitrary point $x$ within the manifold $M$. This quantity is commonly referred to as the "F-length" or simply the "Finslerian length" of the tangent vector $y$ in $T_xM$. If we fix the point $x\in M$ in the manifold $M$, then $F(x,.)$ can eat every tangent vector $y\in T_xM$ and spits a real number.\\
We can think of while considering physical problems involving position and direction,  the Finsler metric $F(x,y)$ as having two arguments $x$ and $y$ representing location and velocity at any point $p\in M$ respectively.\\

\begin{definition}{(Finsler Manifold)}\\
	A differentiable manifold denoted as $M$, when equipped with a Finsler metric  $F(x, y)$, is referred to as a Finsler manifold or Finsler space. This is typically denoted as $(M, F)$.
\end{definition}

\begin{definition}{(Finsler Geometry)}\\
	Consider a Finsler manifold denoted as $(M, F)$, where $F(x, y)$ represents a Finsler metric. The branch of geometry that relies on the Finsler metric $F(x, y)$ defined on the manifold $M$ is termed Finsler geometry.
\end{definition}

In Finsler geometry, a significant and distinct class of Finsler metric known as the $(\alpha, \beta)$-metric is of particular importance. This metric is defined as follows:
 
 \begin{definition}[$(\alpha,\beta)$-metric]
 	
 	Consider a Finsler space denoted as $(M, F(x, y))$, where $F(x, y)$ represents the Finsler fundamental function. This space is said to possess an $(\alpha, \beta)$-metric if the fundamental function $F(x, y)$ can be expressed in the following manner:
 	\begin{align*}
 		F(x,y)=F\left(\alpha(x,y), \beta(x,y)\right)
 	\end{align*}

 	In this expression, $F\left(\alpha(x,y), \beta(x,y)\right)$ is a differentiable function of two variables $\alpha(x, y)$ and $\beta(x, y)$. Here, $\alpha(x, y)$ is the Riemannian fundamental function defined as $\sqrt{a_{ij}(x)y^iy^j}$, where $a_{ij}(x)$ represents a Riemannian metric tensor, and $\beta(x, y)$ is a differential 1-form defined on the tangent bundle $TM$, with $b_i(x)$ representing a covariant vector field.
 	\end{definition}
 
 The class of $(\alpha, \beta)$-metrics was originally introduced by the renowned geometer M. Matsumoto ~\cite{Matsumoto1986}.

 Some important examples of $(\alpha,\beta)$-metric are:

 \begin{example}
 	The  metric defined by 
 	\begin{align*}
 		F(x,y)=\alpha(x,y)+\beta(x,y)
 	\end{align*}
 	is called Randers metric and the space $\left(M,F(x,y)=\alpha(x,y)+\beta(x,y)\right)$ constructed with Randers metric is called Randers space. This metric was first introduced by Physicist G. Randers in 1941 ~\cite{Randers1941}, in his study of general relativity.
 \end{example}

 \begin{example}
 	The metric defined by  
 	\begin{align*}
 		F=\frac{\alpha^2(x,y)}{\beta(x,y)}, \beta>0
 	\end{align*}
 	is called Kropina metric. A space $\left(M,F=\frac{\alpha^2(x,y)}{\beta(x,y)}\right)$ constructed with Kropina metric is called Kropina  space. This metric was introduced by the Russian physicist V.K. Kropina ~\cite{Kropina1961}. It has many important and interesting applications  in physics, electron optics with a magnetic field, dissipative mechanics and irreversible thermodynamics, relativistic field theory, control theory, evolution and developmental biology ~\cite{AntonelliIngardenMatsumoto1993}. 
 	
 \end{example}
 
 \begin{example}
 	The metric defined by  
 	\begin{align*}
 		F=\frac{\alpha^{n+1}(x,y)}{\beta^{n}(x,y)}, \beta>0
 	\end{align*}
 	is called generalized Kropina metric. A space $\left(M,F=\frac{\alpha^{n+1}(x,y)}{\beta^{n}(x,y)}\right)$ constructed with Kropina metric is called generalized  Kropina  space. 
 \end{example}

 \begin{example}
 	
 	The  metric defined by 
 	\begin{align*}
 		F=\frac{\alpha^2(x,y)}{\alpha(x,y)-\beta(x,y)}, \alpha-\beta>0
 	\end{align*}
 	is called Matsumoto metric. A space $\left(M,F=\frac{\alpha^2(x,y)}{\alpha(x,y)-\beta(x,y)}\right)$ constructed with Matsumoto metric is called Matsumoto space. This metric was first introduced by M. Matsumoto ~\cite{Matsumoto1989} while investigating the model of a Finsler space. This metric is also  named as  slope metric.
 \end{example}

 \begin{example}
 	\label{definition1.9}
 	The  metric defined by 
 	\begin{align*}
 		F(x,y)=\frac{[\alpha(x,y)+\beta(x,y)]^2}{\alpha(x,y)}
 	\end{align*}
 	is called square metric. We say the space $\left(M,F(x,y)=\frac{[\alpha(x,y)+\beta(x,y)]^2}{\alpha(x,y)}\right)$ constructed with  square metric  Shen's square  space. 
 \end{example}

 \begin{example}
 	The  metric defined by 
 	\begin{align*}
 		F(x,y)=\frac{[\alpha(x,y)+\beta(x,y)]^{n+1}}{[\alpha(x,y)]^n}
 	\end{align*}
 	is called generalized square metric. We say the space $\left(M,F(x,y)=\frac{[\alpha(x,y)+\beta(x,y)]^{p+1}}{[\alpha(x,y)]^p}\right)$ constructed with generalized square metric the1.generalized square  space or the Shen's generalized square  space. If we put $n=1$ in above metric, i.e., in generalized square metric, we get the  square metric of definition ~\eqref{definition1.9}.   
 \end{example}

 Among these Finsler metrics  this was the \enquote{generalized square metric}   that draws our attention to work with hypersurface  of a Finsler space $(M,F(x,y))$.\\
 Let us first define the meaning of a hypersurface:
 \begin{definition}
 		A hypersurface is an embedded submanifold of codimention 1. That is, a submanifold of dimension less that 1 of a given manifold  is called hypersuraface of the underlying manifold.
 \end{definition} 
This means, if we have a manifold $M$ such that dimension of $M$ is $n$, then the submanifold having dimension $n-1$ will be understood as hypersurface of the given manifold.\\
A submanifold of dimension $n-1$ is generally denoted by the symbol $M^{n-1}$. In this paper, we will also denote a hupersurafce of a manifold $M$ by the symbol $M^{n-1}$.

\begin{example}
	Let $M=R^2$ be a manifold. Then a straight line in Figure  \ref{figure:3}    denoted by $M^{n-1}$  and defined by linear equation in two variables $x$ and $y$, i.e.,  $ax+by+c=0$ is a hypersuarface of the underlying manifold $M=R^2$, because dim($M$)=dim($R^2$)=2 while the   dim($M^{n-1}$)=$n-1=2-1=1$.
\end{example}

\begin{figure}[h!]
	\includegraphics[scale=0.5]{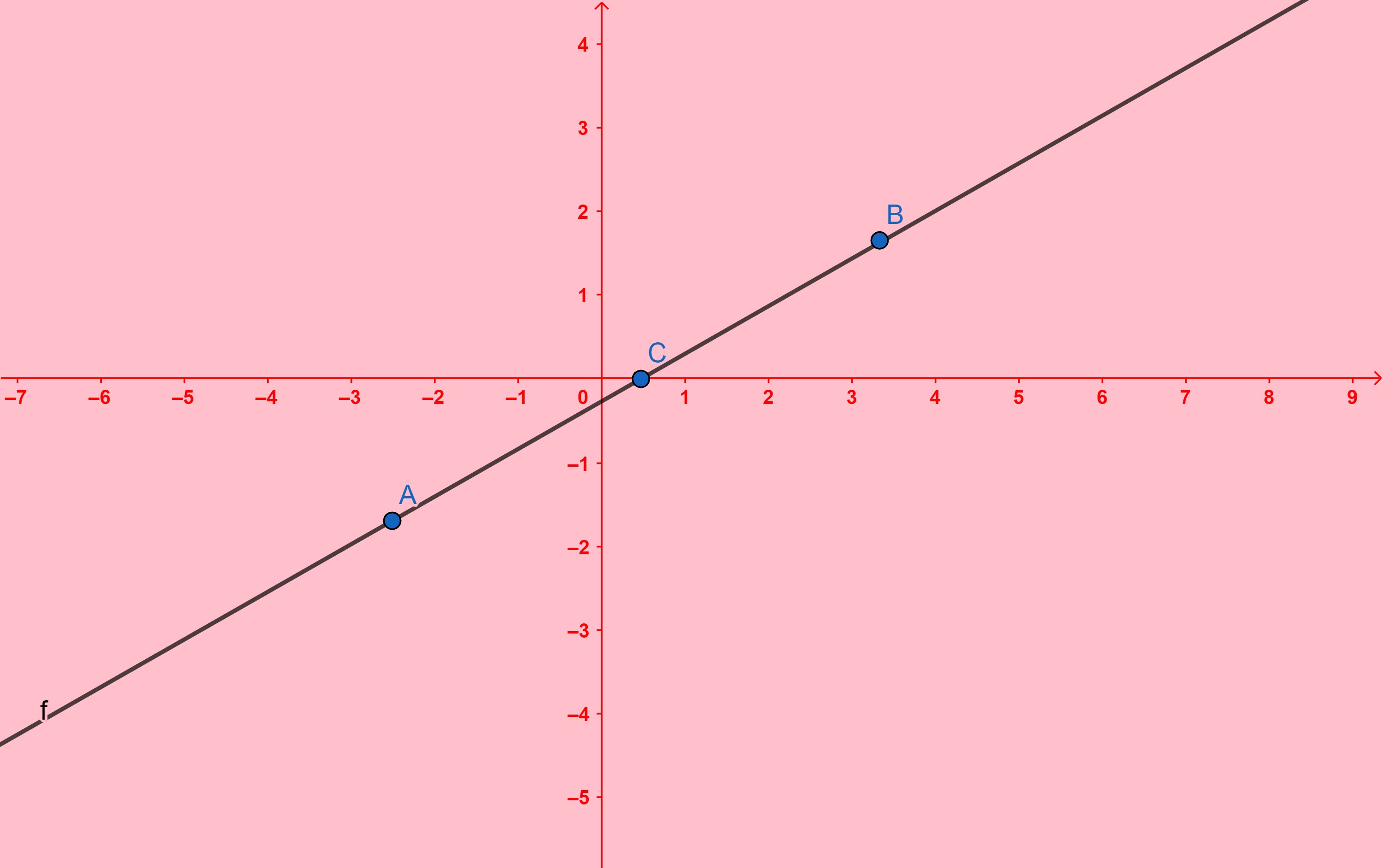}
	\caption{The line represented by eqution ax+by+c=0 is a hypersurface of the 2-dimensional manifold $R^2$.}
	\label{figure:3}
\end{figure}

\begin{example}
	Let $M=R^3$ be a manifold. Then the unit sphere in Figure \ref{figure:4} denoted by $M^{n-1}$  and defined by $x^2+y^2+z^2=1$ is a hypersuarface of the underlying manifold $M=R^3$, because dim($M$)=dim($R^3$)=3 while the   dim($M^{n-1}$)=$n-1=3-1=2$.
\end{example}

\begin{figure}[h!]
	\includegraphics[scale=0.40]{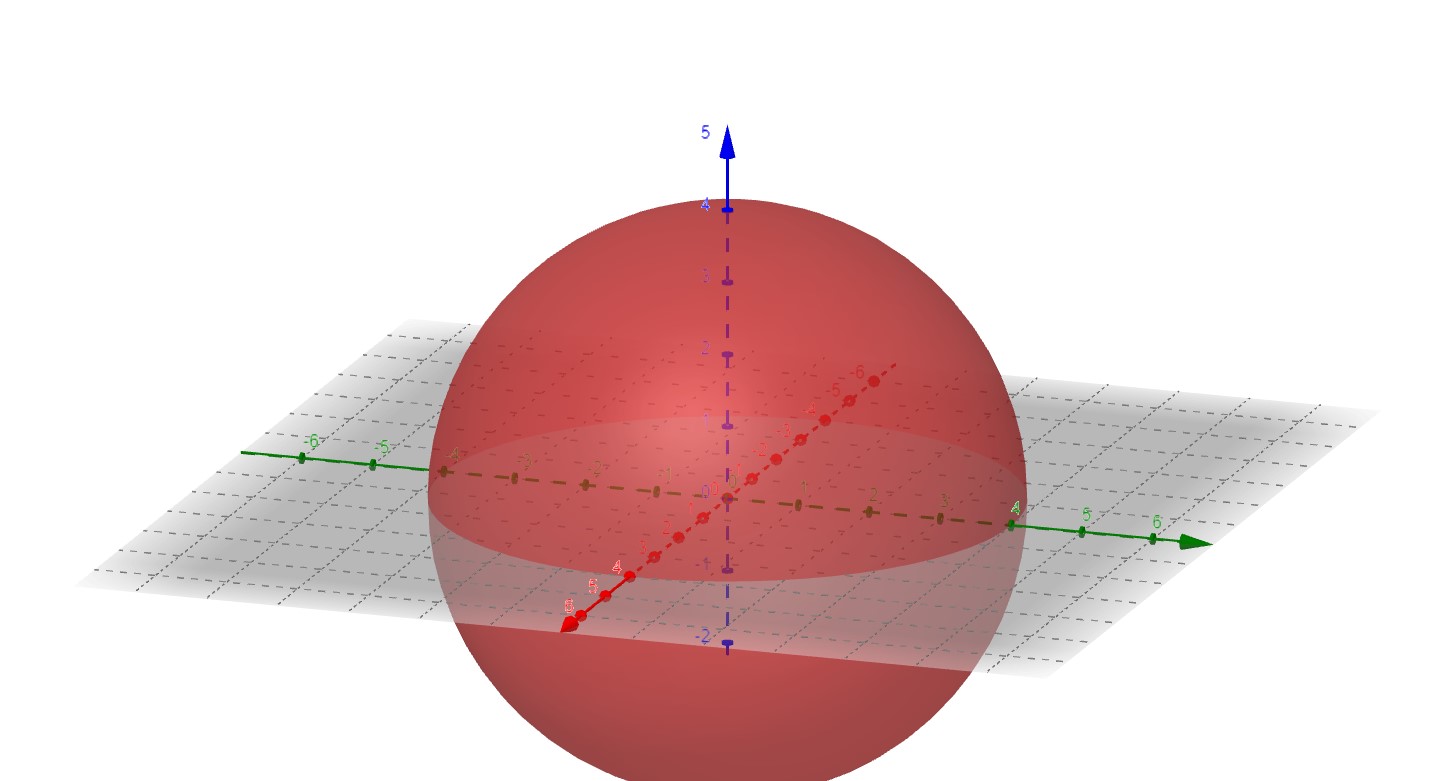}
	\caption{The sphere represented by equation $x^2+y^2+z^2=16$ is a hypersurface of the 3-dimensional manifold $R^3$.}
	\label{figure:4}
\end{figure}

 Matsumoto ~\cite{Matsumoto1985}, a prominent Finslerian, was the first person who studies the hypersurfaces and charecterised the special hypersurfaces $M^{n-1}$ of a Finsler manifold. He, specifically,  characterised the  properties of  hypersurface $M^{n-1}$ of Randers space ~\cite{Randers1941}. After this, the number of Finslerians dramatically increased to show their interest in  Finsler hypersurface $M^{n-1}$. Many authors around the world 
 (~\cite{Kitayama2002},~\cite{Leelee2001}, ~\cite{Shankerchaubeyandey2013},~\cite{Shankerravindra2012},~\cite{Shankersingh2015},~\cite{SinghKumari2001},~\cite{Chaubey2014},~\cite{Panday2018},~\cite{Mishra2017}, ~\cite{Manish2022})  did study the properties of special hypersurface $M^{n-1}$ and derived the conditions under which a Finsler hypersurface $M^{n-1}$ of a Finsler manifold $(M,F(x,y))$ becomes a hyperplane of first kind, second kind but not of the third kind. Aim of the present paper is to investigate the hypersurface $M^{n-1}$ of Finsler space using generalized  square metric  $F(x,y)=\frac{[\alpha(x,y)+\beta(x,y)]^{n+1}}{[\alpha(x,y)]^n}$.

\begin{center}
\hfil \section{Preliminaries}\hfil
\end{center}
\label{sec:2}
We consider the Finsler space $(M,F)$, where $F$ is the generalized square metric, that is  given by
\begin{equation}
\label{eq2.1}
F(\alpha,\beta)=\frac{(\alpha+\beta)^{n+1}}{\alpha^n}
\end{equation}
Calculate all the partial derivatives of equation ~\eqref{eq2.1} up to second order, we get
\begin{align} 
	\label{eq2.2}
F_\alpha&=\frac{(\alpha-n\beta)(\alpha+\beta)^n}{\alpha^{{n+1}}}\\ 
\label{eq2.3}
F_\beta&=\frac{(n+1)(\alpha+\beta)^n}{\alpha^n}\\
\label{eq2.4}
F_{\alpha\alpha}&=\frac{n(n+1)\beta^2(\alpha+\beta)^{n-1}}{\alpha^{n+2}}\\
\label{eq2.5}
F_{\beta\beta}&=\frac{n(n+1)(\alpha+\beta)^{n-1}}{\alpha^n}\\
\label{eq2.6}
F_{\alpha\beta}&=-\frac{n(n+1)\beta(\alpha+\beta)^{n-1}}{\alpha^{n+1}}
\end{align}
We already know that, in a general Finsler manifold $(M,F)$, the normalized element of support $l_i=\frac{\partial F}{\partial y_i}$ and the angular metric tensor $h_{ij}$~\cite{Randers1941} are evaluated by the following formula:
\begin{align}
	\label{eq2.7}
l_i=\frac{F_\alpha y_i}{\alpha}+F_\beta b_i
\end{align}
\begin{align}
\label{eq2.8}
h_{ij}=pa_{ij}+q_0 b_i b_j+q_1(b_iy_j+b_jy_i)+q_2y_i y_j,
\end{align}
and the coefficients are defined and calculated as follows:
\begin{align}
\nonumber
y_i=a_{ij} y^j
\end{align}
 
\begin{align}
\label{eq2.9}
p&=\frac{FF_\alpha}{\alpha}=\frac{(\alpha-n\beta)(\alpha+\beta)^{2n+1}}{\alpha^{2n+2}}\\
\label{eq2.10}
q_0&=FF_{\beta\beta}=\frac{n(n+1)(\alpha+\beta)^{2n}}{\alpha^{2n}}\\
\label{eq2.11}
q_1&=\frac{FF_{\alpha\beta}}{\alpha}=-\frac{n(n+1)\beta(\alpha+\beta)^{2n}}{\alpha^{2n+2}}
\end{align}
\begin{align}
	\nonumber
	q_2&=\frac{F(F_{\alpha\alpha}-\frac{F_\alpha}{\alpha})}{\alpha^2}\\
	\label{eq2.12}
	&=\frac{(\alpha+\beta)^{2n}\left\{n\beta(n+2\beta+n\beta)-\alpha(\alpha+\beta)\right\}}{\alpha^{2n+4}}
\end{align} 

We also know that, in a general Finsler manifold $(M,F)$, the fundamental metric tensor $g_{ij}=\frac{1}{2}\frac{\partial^2 F^2}{\partial y^i\partial y^j}$ is evaluated by ~\cite{Randers1941} the following formula:
\begin{align}
\label{eq2.13}
g_{ij}=pa_{ij}+p_0b_ib_j+p_1(b_iy_j+b_jy_i)+p_2y_iy_j
\end{align}
whereas its coefficients $p,p_0,p_1$ and $p_2$ are defined and calculated as follows: 

\begin{align}
\nonumber
p&=\frac{FF_\alpha}{\alpha}\\
\label{eq2.14}
&=\frac{(\alpha-n\beta)(\alpha+\beta)^{2n+1}}{\alpha^{2n+2}}
\end{align}
\begin{align}
\nonumber
p_0&=q_0+F_\beta^2\\
\label{eq2.15}
&=\frac{(n+1)(2n+1)(\alpha+\beta)^{2n}}{\alpha^{2n}}
\end{align}
\begin{align}
\nonumber
p_1&=q_1+\frac{pF_\beta}{F}\\
\label{eq2.16}
&=\frac{(n+1)(\alpha+\beta)^{2n}(\alpha-2n\beta)}{\alpha^{2n+2}}
\end{align}
\begin{align}
\nonumber
p_2&=q_2+\frac{p^2}{F^2}\\
\label{eq2.17}
&=\frac{\beta(\alpha+\beta)^{2n}\left\{2n^2\beta+2n\beta-n\alpha-\alpha\right\}}{\alpha^2{(n+2)}}
\end{align}

We know that, in a Finsler manifold $(M,F)$, reciprocal metric tensor of a fundamental metric tensor $g_{ij}=\frac{1}{2}\frac{\partial\partial F^2}{\partial y^i\partial y^j}$  is denoted by  $g^{ij}$ and is evaluated by the formula  ~\cite{Randers1941}
\begin{align}
\label{eq2.18}
g^{ij}=\frac{a^{ij}}{p}-S_0b^ib^j-S_1(b^iy^j+b^jy^i)-S_2y^iy^j
\end{align}
whereas its coefficients $b^i$, $S_0$, $S_1$ and $S_2$ are evaluated by the following  formulae:
\begin{align}
\nonumber
b^i&=a^{ij}b_j
\end{align}
\begin{equation}
\label{eq2.19}
S_0=\frac{pp_0+(p_0p_2-p_1^2)\alpha^2}{p\zeta}
\end{equation}
\begin{equation}
\label{eq2.20}
S_1=\frac{pp_1+(p_0p_2-p_1^2)\beta}{p\zeta}
\end{equation}
\begin{equation}
\label{eq2.21}
S_2=\frac{pp_2+(p_0p_2-p_1^2)b^2}{p\zeta}
\end{equation}
\begin{equation}
\label{eq2.22}
 \zeta=p(p+p_0b^2+p_1\beta)+(p_0p_2-p_1^2)(\alpha^2b^2-\beta^2)
\end{equation}

where $b^2=a_{ij}b^ib^j$.\\
Let us define the  $hv$-torsion tensor $C_{ijk}=\frac{1}{2}\frac{\partial g_{ij}}{\partial y^k}$ as follows ~\cite{Shibata1984}:
\begin{align}
\label{eq2.23}
C_{ijk}=\frac{p_1(h_{ij}m_k+h_{jk}m_i+h_{ki}m_j)+\gamma_1m_im_jm_k}{2p}
\end{align}
and its   coefficients $\gamma_1$ and $m_i$ are evaluated by the  formulae  
\begin{align}
\label{eq2.24}
\gamma_1=p\frac{\partial p_0}{\partial\beta}-3p_1q_0, m_i=b_i-\frac{y_i\beta}{\alpha^2}
\end{align}
Here $m_i$ is known as non-zero covariant vector orthogonal to element of support $y^i$.
Let $\Gamma_{jk}^i$be the components of Christoffel symbol of the associated Riemannian space $R^n$ and $\nabla_k$ be the covariant differentiation with respect to $x^k$ relative to Christoffel symbol. Now we put\\
\begin{align}
\label{eq2.25}
2E_{ij}&=b_{ij}+b_{ji}\\
\label{eq2.26}
2F_{ij}&=b_{ij}-b_{ji}
\end{align}
where $b_{ij}=\nabla_jb_i$.
Let $C\Gamma=(\Gamma_{jk}^{*i},\Gamma_{0k}^{*i},C_{jk}^i)$ be Cartan connection of $(M,F)$. The difference tensor $D_{jk}^i=\Gamma_{jk}^{*i}-\Gamma_{jk}^i$ of the special Finsler manifold $(M,F)$ is given by ~\cite{Shibata1984}
  \begin{align}
 \nonumber
 \label{eq2.27}
 D_{jk}^i=& B^iE_{jk}+F_k^iB_j+F_j^iB_k+B_j^ib_{0k}+B_k^ib_{0j}-b_{0m}g^{im}B_{jk}
  -C_{jm}^iA_k^m\\ &-C_{km}^iA_j^m+C_{jkm}A_s^mg^{is}
 + \lambda^s(C_{jm}^iC{sk}^m+C_{km}^iC_{sj}^m-C_{jk}^mC_{ms}^i)
 \end{align}
 where
  \begin{align}
  \label{eq2.28}
 B_k&=p_0b_k+p_1y_k\\
 \label{2.29}
 B^i&=g^{ij}B_j\\
 \label{eq2.30}
 B_{ij}&=\frac{p_1(a_{ij}-\frac{y_iy_j}{\alpha^2})+\frac{\partial p_0}{\partial \beta}m_im_j}{2}\\
 \label{eq2.31}
 B_i^k&=g^{kj}B_{ji}\\
 \label{eq2.32}
 A_k^m&=B_k^mE_{00}+B^mE_{k0}+B_kF_0^m+B_0F_k^m\\
 \label{eq2.33}
 \lambda^m&=B^mE_{00}+2B_0F_0^m\\
 \label{eq2.34}
  F_i^k&=g^{kj}F_{ji}\\
  \label{eq2.35}
 B_0&=B_iY^i
 \end{align}
 Here as well as henceforward '0' denotes tensorial contraction with $y^i$ besides $p_0$,$q_0$ and $S_0.$\textcolor{white}{"}\\
 
 \begin{center}
 	\section{Induced Cartan Connection}
 \end{center}
\label{sec:3}
 	
 Let $(M,F)$ be a Finsler manifold, where $F(\alpha,\beta)=\frac{(\alpha+\beta)^{n+1}}{\alpha^n}$ is generalized square metric. Also, let $M^{n-1}$ be a  hypersurface of the Finsler manifold $(M,F)$ whose hypothetical picture is depicted in the Figure \ref{figure:4}.  
 

\begin{figure}[h!]
	\includegraphics[scale=0.5]{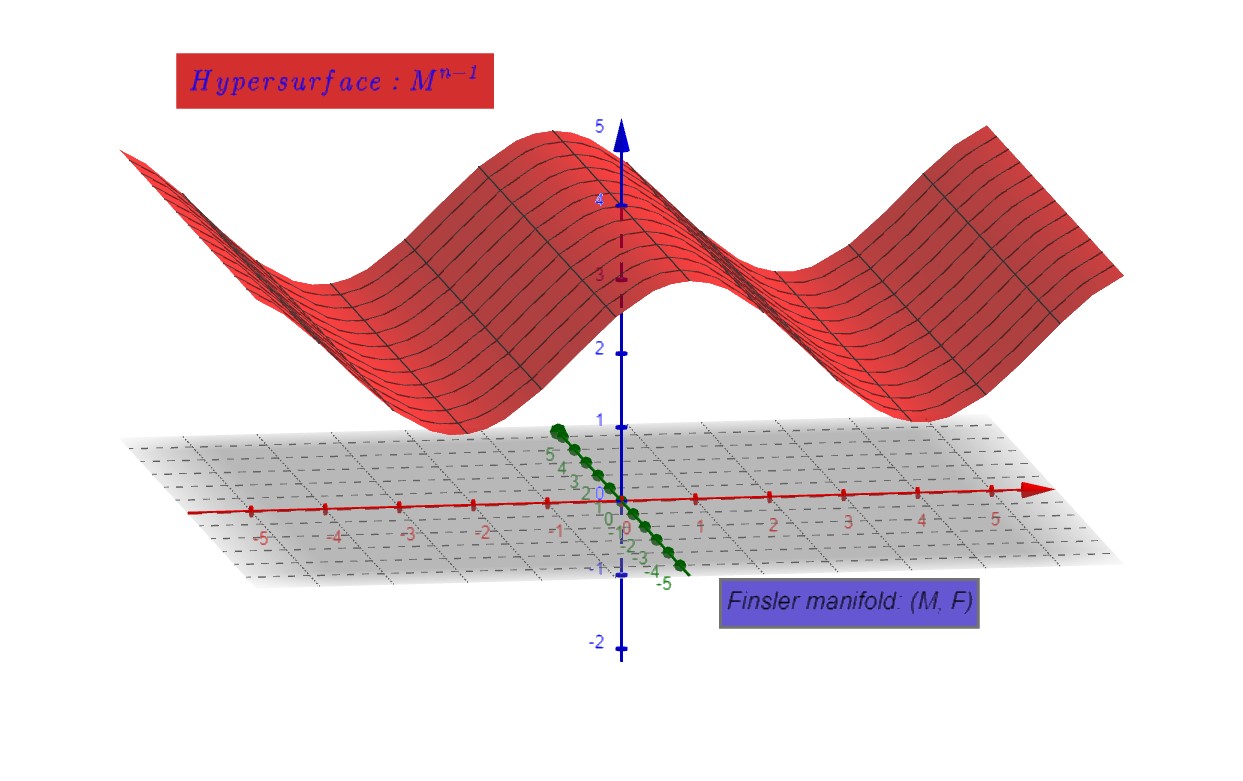}
	\caption{A hypothetical picture of hypersurface $M^{n-1}$ of the 3-dimensional Finsler manifold $(M,F)$.}
	\label{figure:5}
\end{figure}

 We will describe this hypersurface $M^{n-1}$ by following  parametric equations:
 \begin{align}
 	\label{eq3.1}
 	 x^i=x^i(u^\alpha), (i=1,2,3,........n; \alpha=1,2,3,........n-1)
 \end{align}
  
  where $u^\alpha$ is a parameter that  represents coordinates on the hypersurface $M^{n-1}$. 
  Do you know why we are considering parametric equation of the hypersurface?  We considered parametric equation only because of the fact that parametrization always makes our life simpler. Remember those days when we used to consider parametric equations of curves and surfaces. Moreover, using parametrizations we were able to calculate speed of the curves as well as tangent planes and normal lines of the sphere. In elementary differential geometry we also have learnt that parametrization helps us to determine the shape, that is, curvature and torsion  of a curves and surfaces. Thus a parametrization has more information than the set of points constituting the hypersurface $M^{n-1}$.   
  
 Now differentiating the  equation ~\eqref{eq3.1} of the hypersurface with respect to parameters $u^\alpha$, we get $B_\alpha^i=\frac{\partial x^i}{\partial u^\alpha}$. Here each $B_\alpha^i=\frac{\partial x^i}{\partial u^\alpha}$ for ($\alpha$=1,2,3,........n-1) represents components of tangent vectors and these tangent vectors $B_\alpha^i$ represent a tangent space at a point $p$ of the hypersurface $M^{n-1}$ . Let the matrix corresponding to first derivative $B_\alpha^i=\frac{\partial x^i}{\partial u^\alpha}$  be $[B_\alpha^i]=[\frac{\partial x^i}{\partial u^\alpha}]$,  and it has maximal rank , namely, (n-1). The maximal rank required here is to ensure that tangent vectors forms linearly independent set so that any generic  vector tangent to $M^{n-1}$ is linearly expressible in terms of these linearly independent tangent  vectors.   To introduce a Finsler structure in the hypersurface $M^{n-1}$ , the supporting element $y^i$ at a point $u^\alpha$ of $M^{n-1}$ is assumed to be tangential to $M^{n-1}$, so that we may write

  \begin{align}
  \label{eq3.2}
 y^i=B_\alpha^i(u)v^\alpha
 \end{align}
Therefore $v^\alpha$ is the element of support of hypersurface $M^{n-1}$ at the point $u^\alpha$. The metric tensor $g_{\alpha\beta}$ and hv-torsion tensor $C_{\alpha\beta\gamma}$ of hypersurface $M^{n-1}$ are defined by 
\begin{align}
	\label{eq3.3}
	g_{\alpha\beta}=g_{ij}B_\alpha^iB_\beta^j,C_{\alpha\beta\gamma}=C_{ijk}B_\alpha^iB_\beta^jB_\gamma^k
\end{align}

 Now the unit normal vector $N^i(u,v)$ at an arbitrary point $u^\alpha$ of the hypersurface $M^{n-1}$ is defined as follows: 
 \begin{definition}
 	A vector $N^i(u,v)$ at a point $u^\alpha$ of the hypersurface $M^{n-1}$ is said to be unit  normal vector if  
 	\begin{align}
 		\label{eq3.4}
 		g_{ij}(x(u,v),y(u,v))B_\alpha^iN^j=0,g_{ij}(x(u,v),y(u,v))N^iN^j=1
 	\end{align} 
 \end{definition}  
 
 Let us define angular metric tensor $h_{ij}$ as follows:
 \begin{definition}
 	We say the tensor $h_{ij}$ an angular metric tensor, if $h_{ij}$ satisfies the following conditions:
 	\begin{align}
 		\label{eq3.5}
 		h_{\alpha\beta}=h_{ij}B_\alpha^iB_\beta^j, h_{ij}B_\alpha^iN^j=0, h_{ij}N^iN^j=1
 	\end{align}
 \end{definition}

 Let $(B_i^\alpha,N_i)$ be the inverse of $(B_\alpha^i,N^i),$ then we have
 \begin{align*}
 	B_i^\alpha=g^{\alpha\beta}g_{ij}B_\beta^j,B_\alpha^iB_i^\beta=\delta_\alpha^\beta,B_i^\alpha N^i=0,\\B_\alpha^iN_i=0,N_i=g_{ij}N^i,B_i^k=g^{kj}B_{ji},\\B_\alpha^iB_j^\alpha+N^iN_j=\delta_j^i
 \end{align*}
 We shall denote by $\mathcal{I}$$C\Gamma$  the connection of a hypersurface $M^{n-1}$ induced from the Cartan connection $C\Gamma$. The induced Cartan connection $\mathcal{I}$$C\Gamma=(\Gamma_{\beta\gamma}^{*\alpha},G_\beta^\alpha,C_{\beta\gamma}^\alpha)$ on hypersurface $M^{n-1}$ induced from the Cartan's connection $C\Gamma=(\Gamma_{jk}^{*i},\Gamma_{0k}^{*i},C_{jk}^i)$ is given by ~\cite{Matsumoto1985}
 \begin{align}
 	\nonumber
 	\Gamma_{\beta\gamma}^{*\alpha}&=B_i^\alpha(B_{\beta\gamma}^i+\Gamma_{jk}^{*i}B_\beta^jB_\gamma^k)+M_\beta^\alpha H_\gamma\\
 	\nonumber G_\beta^\alpha&=B_i^\alpha(B_{0\beta}^i+\Gamma_{0j}^{*i}B_\beta^j)\\
 	\label{eq3.6}
 	C_{\beta\gamma}^\alpha&=B_i^\alpha C_{jk}^iB_\beta^jB_\gamma^k
 \end{align}
 where second fundamental $v$-tensor $M_{\beta\gamma}$ is defined by  
 \begin{align}
 	\nonumber
 	M_{\beta\gamma}&=N_iC_{jk}^iB_\beta^jB_\gamma^k\\
 	\nonumber      &=N_ig^{li}C_{ljk}B_\beta^jB_\gamma^k \quad\text{(using $C_{jk}^i=g^{li}C_{ljk}$)} \\
 	\nonumber      &=N^lC_{ljk}B_\beta^jB_\gamma^k       \quad\text{(using $N^ig^{li}=N^l$)}           \\
 	\label{eq3.7}    &=C_{ijk}B_\beta^iB_\gamma^jN^k      \quad\text{(adjusting the indices $j$, $k$ and $l$)}          \\
 	\nonumber    M_\beta^\alpha&=g^{\alpha\gamma}M_{\beta\gamma}
 \end{align}
 
 and normal curvature vector $H_\beta$ is defined by 
 
 \begin{align}
 	\nonumber H_\beta&=N_i(B_{0\beta}^i+\Gamma_{0j}^{*i}B_\beta^j),
 \end{align}
 where
 \begin{align}\nonumber
 	B_{\beta\gamma}^i&=\frac{\partial B_\beta^i}{\partial U^\gamma}\\
 	\nonumber B_{0\beta}^i&=B_{\alpha\beta}^iv^\alpha
 \end{align}
 The quantities $M_{\beta\gamma}$ and $H_\beta$ appeared in above equations are called the second fundamental $v$-tensor and normal curvature vector respectively ~\cite{Matsumoto1985}. The second fundamental $h$-tensor $H_{\beta\gamma}$ is defined as ~\cite{Matsumoto1985}
 \begin{align}
 	\label{eq3.8}
 	H_{\beta\gamma}=N_i(B_{\beta\gamma}^i+\Gamma_{jk}^{*i}B_\beta^jB_\gamma^k)+M_\beta H_\gamma
 \end{align}
 where
 \begin{align}
 	\label{eq3.9}
 	M_\beta=C_{jk}^iB_\beta^jN_iN^k=C_{ijk}B_\beta^iN^jN^k
 \end{align}
 The relative $h$-covariant derivative and  $v$-covariant derivative of projection factor $B_\alpha^i$ with respect to induced Cartan connection $\mathcal{I}$$C\Gamma$ are respectively given by 
 \begin{align}
 	\label{eq3.10}
 	B_{\alpha|\beta}^i=H_{\alpha\beta}N^i\\
 	\label{3.11} B_{\alpha|\beta}^i=M_{\alpha\beta}N^i
 \end{align}
 The equation ~\eqref{eq3.8} shows that $H_{\beta\gamma}$ is not always symmetric and 
 \begin{align}
 	\label{eq3.12}
 	H_{\beta\gamma}-H_{\gamma\beta}=M_\beta H_\gamma-M_\gamma H_\beta
 \end{align}
 Thus the above equation simplifies to 
 \begin{align}
 	\label{eq3.13}
 	H_{0\gamma}=H_\gamma, H_{\gamma0}=H_\gamma+M_\gamma H_0
 \end{align}

 \begin{definition}[~\cite{Matsumoto1985}] A hypersurface $M^{n-1}$ of a  Finsler manifold $(M,F)$ is said to be a hyperplane of first kind if each path connecting two different points of the hypersurface $M^{n-1}$ with respect to the induced Cartan connection $\mathcal{I}$$C\Gamma$ is also becomes the  path of the ambient Finsler manifold $(M,F)$ with respect to Cartan connection $C\Gamma$.
 \end{definition}
 
 \begin{definition}[~\cite{Matsumoto1985}]  A hypersurface $M^{n-1}$ of a  Finsler manifold $(M,F)$ is said to be a hyperplane of second kind if each $h$-path of the hypersurface $M^{n-1}$ with respect to the induced Cartan connection $\mathcal{I}$$C\Gamma$ is also the $h$-path of the ambient Finsler manifold $(M,F)$ with respect to Cartan connection $C\Gamma$.
 \end{definition}
 
 \begin{definition}[~\cite{Matsumoto1985}]  A hypersurface $M^{n-1}$ of a  Finsler manifold $(M,F)$ is said to be a hyperplane of third kind if unit normal vector $B^i$ of the hypersurface $M^{n-1}$ with respect to the metric $F$ is parallel along each curve $(u^\alpha, v^\alpha)$. 
 \end{definition}
 
 If one wants to prove a hypersurface a hyperplane of first kind, hyperplane of second kind and hyperplane of third kind, it is very difficult to prove it only  with the help of definitions mentioned above, in that situation one should  incorporate below sufficient conditions given by M. Matsumoto to prove the same:
 \begin{lemma}[~\cite{Matsumoto1985}]
 	\label{lemma3.6}
 	The normal curvature $H_0=H_\beta v^\beta$ vanishes if and only if normal curvature vector $H_\beta$ vanishes.
 \end{lemma}
 \begin{lemma}[~\cite{Matsumoto1985}]
 	\label{lemma3.7}
 	A hypersurface $M^{n-1}$is a hyperplane of first kind if and only if $H_\alpha=0$.
 \end{lemma}
 \begin{lemma}[~\cite{Matsumoto1985}]
 	\label{lemma3.8}
 	A hypersurface $M^{n-1}$ is a hyperplane of second kind with respect to Cartan connection $C\Gamma$ if and only if $H_\alpha=0$ and $H_{\alpha\beta}=0$.	
 \end{lemma}
 \begin{lemma}[~\cite{Matsumoto1985}]
 	\label{lemma3.9}
 	A hypersurface $M^{n-1}$ is a hyperplane of third kind with respect to Cartan connection $C\Gamma$ if and only if $H_\alpha=0,$ $H_{\alpha\beta}=0$ and $M_{\alpha\beta}=0$.
 \end{lemma}

\begin{center}
\hfil \
\section{Hypersurface $M^{n-1}$ of the special Finsler space}\hfil
\end{center}
\label{sec:4}

We know that a hypersurface is an embedded subspace of codimension 1. That is, a subspace of dimension less that 1 than the dimension of a given manifold is called hypersuraface of the given  manifold or space. For the sake of our study we are dealing with hypersurfaces rather than the  subspaces of arbitrary dimensions. In this paper  we are specifically confined to Finslerian  hypersurfaces $M^{n-1}$. \\
Let us proof the following propositions in context of Finslerian hypersurfaces.
\begin{proposition}
	Let $(M,F)$ be a Finsler manifold, where $F(\alpha,\beta)=\frac{(\alpha+\beta)^{n+1}}{\alpha^n}$, $n\in N$, is a generalized square metric and  $M^{n-1}$ be its hypersurface.  Then fundamental function of the hypersurface  $M^{n-1}$ induced from the Finsler manifold $(M,F)$ is a Riemannian metric.
\end{proposition}
\begin{proof}
	
It is given that $(M,F)$ be a Finsler manifold, where $F(\alpha,\beta)=\frac{(\alpha+\beta)^{n+1}}{\alpha^n}$, $n\in N$, is a generalized square metric.  Let level equation of the hypersurface  $M^{n-1}$ be given by
 
 \begin{equation*}
 	b(x)=c
 \end{equation*}
 
  where c is a  real number. \\
  Take the gradient of the above level equation  representing hypersurface $M^{n-1}$, we get  
  \begin{equation*}
  	b_i(x)=\partial_ib
  \end{equation*}
  
   Again consider the parametric equation of the same hypersurface $M^{n-1}$ as 
   
   \begin{equation*}
   	x^i=x^i(u^\alpha)
   \end{equation*}
   
   Differentiating the equation of hypersurface $b(x(u))=c$ with respect to parameter $u^\alpha$, we get
   \begin{align*}
   	\frac{\partial b(x(u))}{\partial x^i}\frac{\partial x^i}{\partial u^\alpha}=0\\
   	b_i(x) B^i_\alpha=0
   \end{align*}
     where   $b_i(x)=\frac{\partial b(x(u))}{\partial x^i}$ and $B^i_\alpha=\frac{\partial x^i}{\partial u^\alpha}$\\
     
      This implies that $b_i(x)$ are normal vector field (covariant component) of hypersurface $M^{n-1}$.\\

Thus at any point of the hypersurface  $M^{n-1}$ we now have
\begin{align}
 \label{eq4.1}
b_iB_\alpha^i=0\\
 \label{eq4.2}
b_iy^i=0 ,i.e., \beta=0
\end{align}

Now, we will see how generalized square metric $F=\frac{(\alpha+\beta)^{n+1} }{\alpha^n }$, $n\in N$, induces a metric on the hypersurface $M^{n-1}$. In this case we will denote induced metric  by $\bar{F}$.
First consider the generalized square metric 
\begin{align*}
	F&=\frac{(\alpha+\beta)^{n+1} }{\alpha^n }\\
	&=\frac{(\sqrt{a_{ij}y^iy^j}+b_iy^i)^{n+1} }{(\sqrt{a_{ij}y^iy^j})^n}\\ 
	&=\frac{\left(\sqrt{a_{ij}B_\alpha^i(u)v^\alpha B_\beta^j(u)v^\beta}+b_iy^i\right)^{n+1}}{\left(\sqrt{a_{ij}B_\alpha^i(u)v^\alpha B_\beta^j(u)v^\beta}\right)^n }\\ 
	&=\frac{\left(\sqrt{a_{ij}B_\alpha^i(u) B_\beta^j(u)v^\alpha v^\beta}+b_iy^i\right)^{n+1}}{\left(\sqrt{a_{ij}B_\alpha^i(u) B_\beta^j(u)v^\alpha v^\beta}\right)^n }
\end{align*}

which is the general induced metric on the corresponding hypersurface  $M^{n-1}$.  Using equation  ~\eqref{eq4.2},  general induced metric of the hypersurface becomes
\begin{align}
	\nonumber F(u,v)&=\frac{\left(a_{ij}B_\alpha^i(u) B_\beta^j(u)v^\alpha v^\beta\right)^{\frac{n+1}{2}}}{\left(a_{ij}B_\alpha^i(u) B_\beta^j(u)v^\alpha v^\beta\right)^{\frac{n}{2}} }\\
	\nonumber&=\sqrt{a_{ij}B_\alpha^i(u)B_\beta^j(u)v^\alpha v^\beta}\\
	\label{eq4.3}   \bar{F}(u,v)&=\sqrt{a_{\alpha\beta}v^\alpha v^\beta}
\end{align} 
where $a_{\alpha\beta}=a_{ij}B_\alpha^i(u)B_\beta^j(u)$.\\
Thus function represented by equation ~\eqref{eq4.3} is  fundamental function or the metric of the hypersurface $M^{n-1}$ induced from the ambient Finsler manifold $(M,F)$.

The fundamental function of the hypersurface $M^{n-1}$ represented by equation ~\eqref{eq4.3} do not have $\beta$ component as  $\beta=b_iy^i=0$ over the hypersurface  $M^{n-1}$  therefore  fundamental function of the hypersurface  $M^{n-1}$ induced from the Finsler manifold $(M,F)$ is a Riemannian metric.
\end{proof}

\begin{proposition}
	Let $(M,F)$ be a Finsler manifold, where $F(\alpha,\beta)=\frac{(\alpha+\beta)^{n+1}}{\alpha^n}$, $n\in N$, is a generalized square metric and  $M^{n-1}$ be its associated  hypersurface. Then the covariant and contravariant components  of normal vector field on the hypersurface $M^{n-1}$ are given by  
	
	\begin{enumerate}
		\item $b_i=\sqrt{\frac{b^2}{1+n(n+1)}}N_i$
		\item $b^i=\sqrt{b^2\left\{1+n(n+1)\right\}}N^i+\frac{b^2}{\alpha}y^i$
	\end{enumerate}
	
	\end{proposition}

\begin{proof}
	It is given that $M^{n-1}$ is a hypersurface of the manifold $(M,F)$, where $F(\alpha,\beta)=\frac{(\alpha+\beta)^{n+1}}{\alpha^n}$, $n\in N$, is a generalized square metric.
	
	Moreover, we know from equation ~\eqref{eq4.2} that $\beta=0$ over the hypersurface $M^{n-1}$. Let us calculate the value of $p$, $p_0$, $p_1$ and  $p_2$. For that,  substitute the value of $\beta=0$ into equations ~\eqref{eq2.14}, ~\eqref{eq2.15}, ~\eqref{eq2.16},  and ~\eqref{eq2.17}, we get
	
	\begin{align}
		\label{eq4.4}
		p=1, p_0=(n+1)(2n+1), p_1=\frac{n+1}{\alpha}, p_2=0
	\end{align}
	Now put the values of $p$, $p_0$, $p_1$, $p_2$ into equations ~\eqref{eq2.19}, ~\eqref{eq2.20},  ~\eqref{eq2.21} and ~\eqref{eq2.22}, we get
	\begin{align}
		\label{eq4.5}
		S_0&=\frac{n(n+1)}{1+n(n+1)b^2}\\
		\label{eq4.6}
		S_1&=\frac{n+1}{\alpha \left\{1+n(n+1)b^2\right\}}\\
		\label{eq4.7}
		S_2&=-\frac{(n+1)^2b^2}{\alpha^2 \left\{1+n(n+1)b^2\right\}}\\
		\label{eq4.8}
	  \zeta&=1+n(n+1)b^2
	\end{align}
	Substituting the  values of $p,S_0,S_1, S_2$ from the equations ~\eqref{eq4.4}, ~\eqref{eq4.5},  ~\eqref{eq4.6} and ~\eqref{eq4.7}  into equation ~\eqref{eq2.18}, we have
	\begin{align}
		\nonumber
	g^{ij}&=\frac{a^{ij}}{1}-\frac{n(n+1)}{1+n(n+1)b^2}\times b^ib^j-\frac{n+1}{\alpha \left\{1+n(n+1)b^2\right\}}\times (b^iy^j+b^jy^i)+\\
\label{eq4.9}	&\frac{(n+1)^2b^2}{\alpha^2 \left\{1+n(n+1)b^2\right\}}\times y^iy^j
	\end{align}
	Multiplying equation ~\ref{eq4.9} by $b_ib_j$ and using   $\beta=b_iy^i=0$, over the hypersurface $M^{n-1}$, it becomes
	\begin{align*}
		g^{ij}&=a^{ij} \times b_ib_j -\frac{n(n+1)}{1+n(n+1)b^2}\times b^ib^j \times b_ib_j-
			\frac{n+1}{\alpha \left\{1+n(n+1)b^2\right\}}\times (b^iy^j+b^jy^i) \times b_ib_j+\\
			&\frac{(n+1)^2b^2}{\alpha^2 \left\{1+n(n+1)b^2\right\}}\times y^iy^j \times b_ib_j\\ \\
	    	&=(a^{ij}b_i)b_j -\frac{n(n+1)}{1+n(n+1)b^2}\times (b^ib_i)(b^jb_j)-\\
			&\frac{n+1}{\alpha \left\{1+n(n+1)b^2\right\}}\times \left\{(b^ib_i)(b_jy^j)+(b_ib^j)(b_iy^i)\right\} + 
			\frac{(n+1)^2b^2}{\alpha^2 \left\{1+n(n+1)b^2\right\}}\times \left\{(b_iy^i)(b_jy^j)\right\} \\ 
			&=b^jb_j -\frac{n(n+1)}{1+n(n+1)b^2}\times (b^2)(b^2)-\\
			&\frac{n+1}{\alpha \left\{1+n(n+1)b^2\right\}}\times \left\{(b^2)(0)+(b^2)(0)\right\} +  
			\frac{(n+1)^2b^2}{\alpha^2 \left\{1+n(n+1)b^2\right\}}\times \left\{(0)(0)\right\} \\ \\
			&=b^2 -\frac{n(n+1)}{1+n(n+1)b^2}\times b^4\\ 
			g^{ij}b_ib_j&=\frac{b^2}{1+n(n+1)}
	\end{align*}
	Thus at any generic point  of the hypersurface $M^{n-1}$, we have 
	\begin{align*}
			g^{ij}b_ib_j&=\frac{b^2}{1+n(n+1)}
	\end{align*}
	
	Now from the above equation and using the equation  ~\eqref{eq3.4}, we get                     
	\begin{align}
		\label{eq4.10}
		 b_i=\sqrt{\frac{b^2}{1+n(n+1)}}N_i
	\end{align}
which is the covariant component of the normal vector field on the hypersurface $M^{n-1}$.
	Now from ~\eqref{eq4.9} and ~\eqref{eq4.10} we get
	
	\begin{align}
		\nonumber
		b^i&=a^{ij}b_j\\
		\label{eq4.11}
		&=\sqrt{b^2\left\{1+n(n+1)\right\}}N^i+\frac{b^2}{\alpha}y^i
	\end{align}
	which is the contravariant component of the normal vector field on the hypersurface $M^{n-1}$.
\end{proof}

\begin{proposition}
	Let $(M,F)$ be a Finsler manifold, where $F(\alpha,\beta)=\frac{(\alpha+\beta)^{n+1}}{\alpha^n}$, $n\in N$, is a generalized square metric and  $M^{n-1}$ be its associated  hypersurface. Then second fundamental v-tensor of hypersurface $M^{n-1}$ is given by \\
	
		 $M_{\alpha \beta}=\frac{(n+1)}{2\alpha}\left(\sqrt{\frac{b^2}{1+n(n+1)}}\right)h_{\alpha\beta}$\\
		 
		 and second fundamental h-tensor $H_{\alpha\beta}$ is symmetric, i.e., $H_{\alpha\beta}=H_{\beta\alpha}$.
	
	 \end{proposition}
\begin{proof}
	
		It is given that $M^{n-1}$ is a hypersurface of the manifold $(M,F)$, where $F(\alpha,\beta)=\frac{(\alpha+\beta)^{n+1}}{\alpha^n}$, $n\in N$, is a generalized square metric.
	
	Moreover, we know from equation ~\eqref{eq4.2} that $\beta=0$ over the hypersurface $M^{n-1}$. Put the value of $\beta=0$ into equations ~\eqref{eq2.14}, ~\eqref{eq2.15}, ~\eqref{eq2.16},  and ~\eqref{eq2.17}, we get
	
	\begin{align*}
	p=1, p_0=(n+1)(2n+1), p_1=\frac{n+1}{\alpha}, p_2=0
	\end{align*}

Now,	put the values of $p,p_0,p_1$ and $p_2$ obtained above into equation ~\ref{eq2.13}, we get fundamental metric tensor of the hypersurace $M^{n-1}$ 
	\begin{align}
		\label{eq4.12}
		g_{ij}=a_{ij}+(n+1)(2n+1)b_ib_j+\frac{(n+1)}{\alpha}(b_iy_j+b_jy_i)
	\end{align}
Let us calculate the value of $q_0$, $q_1$ and  $q_2$. For that,  substitute the value of $\beta=0$ into equations  ~\eqref{eq2.10}, ~\eqref{eq2.11},  and ~\eqref{eq2.12}, we get

	\begin{align*}
	q_0=n(n+1), q_1=0, q_2=-\frac{1}{\alpha^2}
\end{align*}

Substituting the values of  $p,q_0,q_1$ and $q_2$ in equation ~\ref{eq2.8}, we get angular metric tensor of the hypersurface $M^{n-1}$  
	
	\begin{align}
		\label{eq4.13}
		h_{ij}= a_{ij}+n(n+1)b_ib_j-\frac{1}{\alpha^2}y_iy_j
	\end{align}

	Differentiating equation~\ref{eq2.15} with respect to $\beta$, we have
	
	\begin{align*}
		\frac{\partial p_0}{\partial \beta}=\frac{2n(n+1)(2n+1)(\alpha+\beta)^{2n-1}}{\alpha^{2n}}.
	\end{align*}

We know from equation ~\eqref{eq4.2} that $\beta=0$ over the hypersurface $M^{n-1}$ so put the value of $\beta=0$ into above equation and equation ~\ref{eq2.24},   we get

	\begin{align*}
		\frac{\partial p_0}{\partial \beta}&=\frac{2n(n+1)(2n+1)}{\alpha}\\
		\gamma_1&=\frac{n(n^2-1)}{\alpha}\\
		m_i&=b_i
	\end{align*}

	Using the values of $p, p_1,\gamma_1$ and $m_i$ in equation~\ref{eq2.23}, hv-torsion tensor on the hypersurface  $M^{n-1}$, becomes
	
	\begin{align}
		\label{eq4.14}
		C_{ijk}=\frac{(n+1)\left[(h_{ij}b_k+h_{jk}b_i+h_{ki}b_j)+n(n^2-1)b_ib_jb_k\right]}{2\alpha}
	\end{align}
	Substituting the value of $C_{ijk}$ from equation ~\ref{eq4.14} in equation ~\ref{eq3.7} as follows:

	\begin{align*}
		\nonumber &M_{\beta\gamma}=C_{ijk}B_\beta^iB_\gamma^jN^k\\
		\nonumber &\therefore M_{\alpha \beta}=C_{ijk}B_\alpha^iB_\beta^jN^k\\
		\nonumber &=\left[\frac{(n+1)\left(h_{ij}b_k+h_{jk}b_i+h_{ki}b_j)+n(n^2-1)b_ib_jb_k\right]}{2\alpha}\right] B_\alpha^iB_\beta^jN^k\\
		\nonumber  &=\left[\frac{(n+1)\left\{h_{ij}B_\alpha^iB_\beta^jb_k+h_{jk}B_\beta^j(b_iB_\alpha^i)+h_{ki}B_\alpha^j(b_jB_\beta^j)\right\}+n(n^2-1)(b_iB_\alpha^i)(b_jB_\beta^j)b_k}{2\alpha}\right] N^k
	\end{align*}

	 Using the equation ~\ref{eq4.1} in above expression,  we get
	\begin{align*}
		\nonumber &=\left[\frac{(n+1)\left\{h_{ij}B_\alpha^iB_\beta^jb_k+h_{jk}B_\beta^j(0)+h_{ki}B_\alpha^j(0)\right\}+n(n^2-1)(0)(0)b_k}{2\alpha}\right] N^k \\
		\nonumber &=\left[\frac{(n+1)h_{ij}B_\alpha^iB_\beta^jb_k}{2\alpha} \right] N^k\\
		\nonumber &=\left[\frac{(n+1)(h_{ij}B_\alpha^iB_\beta^j)b_k N^k}{2\alpha} \right] \\
	\end{align*}
	
	Using the equation ~\ref{eq3.5},   we get

		\begin{align*}
		\nonumber &=\left[\frac{(n+1)h_{\alpha\beta}b_k N^k}{2\alpha} \right] 
	\end{align*}

	Using the equation ~\ref{eq4.10} in above expression,  we get

\begin{align*}
		\nonumber &=\frac{(n+1)}{2\alpha}\left(\sqrt{\frac{b^2}{1+n(n+1)}}N_k\right)h_{\alpha\beta}N^k    \\
		\nonumber &=\frac{(n+1)}{2\alpha}\left(\sqrt{\frac{b^2}{1+n(n+1)}}\right)h_{\alpha\beta}N_kN^k  
\end{align*}
		
	We know that $N_kN^k=1$. Use this fact in above expression, we get
		\begin{align}
		\label{eq4.15} M_{\alpha \beta}&=\frac{(n+1)}{2\alpha}\left(\sqrt{\frac{b^2}{1+n(n+1)}}\right)h_{\alpha\beta}                
	\end{align}

	Again, substituting the value of $C_{ijk}$ from equation  ~\ref{eq4.14} into equation  ~\ref{eq3.9} as follows:

	\begin{align}
	\nonumber	&M_\beta=C_{ijk}B_\beta^iN^jN^k\\
	\nonumber	&\therefore M_\alpha=C_{ijk}B_\alpha^iN^jN^k\\
	\nonumber	&=\left[\frac{(n+1)\left(h_{ij}b_k+h_{jk}b_i+h_{ki}b_j)+(n^2-1)b_ib_jb_k\right]}{2\alpha}\right] B_\alpha^iN^jN^k
\end{align}
\begin{align}
	\nonumber	&=\left[\frac{(n+1)\left(h_{ij}b_kB_\alpha^i N^jN^k+h_{jk}b_iB_\alpha^i N^jN^k+h_{ki}b_j B_\alpha^iN^jN^k )+n(n^2-1)b_i B_\alpha^i b_jb_k N^jN^k\right]}{2\alpha}\right] 
\end{align}
Using equations ~\ref{eq4.1} and ~\ref{eq3.5} in above expression, we get
\begin{align}
		\label{eq4.16}
		M_\alpha=0
	\end{align}
	Substituting the value of $M_\alpha$ from the equation ~\ref{eq4.16} in equation ~\ref{eq3.12} as follows:  
	\begin{align*}
		H_{\beta\gamma}-H_{\gamma\beta}&=M_\beta H_\gamma-M_\gamma H_\beta\\
	H_{\beta\gamma}-H_{\gamma\beta}&=0 \times H_\gamma-0 \times H_\beta\\
		H_{\beta\gamma}-H_{\gamma\beta}&=0 \\
		H_{\beta\gamma}=H_{\gamma\beta}& \\
\therefore	H_{\alpha\beta}=H_{\beta\gamma}&
	\end{align*}

which shows that  $H_{\alpha\beta}$ is symmetric.
\end{proof}

\begin{theorem}
	\label{theorem4.4}
		Let $(M,F)$ be a Finsler manifold, where $F(\alpha,\beta)=\frac{(\alpha+\beta)^{n+1}}{\alpha^n}$, $n\in N$, is a generalized square metric and  $M^{n-1}$ be its associated  hypersurface. Then the hypersurface $M^{n-1}$ will be  hyperplane of first kind if and only if
	\begin{align*}
			2b_{ij}=b_ic_j+b_jc_i
	\end{align*}

	Moreover we show that  second fundamental tensor $H_{\alpha\beta}$ of $M^{n-1}$ is proportional to it's angular metric tensor $h_{\alpha\beta}$. That is, $	H_{\alpha\beta}=\frac{c_0b}{\sqrt{1+n(n+1)}}h_{\alpha\beta}$.
\end{theorem}

 \begin{proof}
 	Let us differentiate equation ~\ref{eq4.1} with respect to $\beta$, we get
 	\begin{align}
 		\label{eq4.17}
 		b_{i|\beta}B_\alpha^i+b_iB_{\alpha|\beta}^i=0
 	\end{align}

 	Put the value of $B_{\alpha|\beta}^i$ from equation ~\ref{eq3.10} and $b_{i|\beta}=b_{i|j}B_\beta^j+b_{i|j}N^jH_\beta$ into  equation ~\ref{eq4.17}, we get
 	\begin{align}
 		\label{eq4.18}
 		b_{i|j}B_\beta^jB_\alpha^i+b_{i|j}N^jH_\beta B_\alpha^i+b_iH_{\alpha\beta}N^i=0
 	\end{align} 
 	We know that 
 	\begin{align*}
 		b_{i|j}=-b_hC_{ij}^h
 	\end{align*}                                                                                                 
 Put the value of $b_h$ from equation ~\ref{eq4.10} into above expression as follows:
 \begin{align*}
 	b_{i|j}&=-b_hC_{ij}^h\\
 		b_{i|j}&=-\sqrt{\frac{b^2}{1+n(n+1)}}N_hC_{ij}^h\\
 		b_{i|j}B_\alpha^iN^j&=-\sqrt{\frac{b^2}{1+n(n+1)}}N_hC_{ij}^h B_\alpha^iN^j\\
 		                    &=-\sqrt{\frac{b^2}{1+n(n+1)}}M_\alpha\\
 		                    &=-\sqrt{\frac{b^2}{1+n(n+1)}}\times 0        \quad\text{(using Equation ~\eqref{eq4.16} )}\\ 
 		                    &=0
 \end{align*}
 	  	Using $b_{i|j}B_\alpha^iN^j=0$ and equation ~\ref{eq4.10} in the equation ~\ref{eq4.18} and then using the fact that $N_iN^i=1$, we get
 	\begin{align}
 		\label{eq4.19}
 		b_{i|j}B_\beta^jB_\alpha^i+\sqrt{\frac{b^2}{1+n(n+1)}}H_{\alpha\beta}=0
 	\end{align}
 	It is obvious that $b_{i|j}$ is symmetric. Now contracting ~\ref{eq4.19} with $v^\beta$ first and then with $v^\alpha$ respectively and using the equations ~\ref{eq3.2}, ~\ref{eq3.13} and ~\ref{eq4.16}, 
 	we get
 	
 	\begin{align}
 		\label{eq4.20}
 		b_{i|j}B_\alpha^iy^j+\sqrt{\frac{b^2}{1+n(n+1)}}H_{\alpha}=0\\
 	\label{eq4.21}	b_{i|j}y^iy^j+\sqrt{\frac{b^2}{1+n(n+1)}}H_0=0
 	\end{align}
 	We know from the  Lemma \ref{lemma3.6} and Lemma \ref{lemma3.7}, a hypersurface $M^{n-1}$  is a hyperplane of first kind if and only if normal curvature vanishes, i.e., $H_0=0$. Using the value  $H_0=0$ in equation ~\ref{eq4.21} we find that hypersurface  $M^{n-1}$ is  a hyperplane of first kind if and only if $b_{i|j}y^iy^j=0$. This $b_{i|j}$ is the covariant derivative of with respect to Cartan connection $C\Gamma$ of Finsler space $F$, it may depend on $y^i$. Moreover $\nabla_jb_i=b_{ij}$ is the covariant derivative of $b_i$ with respect to Riemannian connection $\Gamma_{jk}^i$ constructed from $a_{ij}(x)$, therefore $b_{ij}$ dose not depend on $y^i$. We shall consider the difference $b_{i|j}-b_{ij}$ of above covariant derivatives in further discussion. The difference tensor $D_{jk}^i=\Gamma_{jk}^{*i}-\Gamma_{jk}^i$ is given by equation ~\ref{eq2.27}. Since $b_i$ is a gradient vector, from equations ~\ref{eq2.25} and ~\ref{eq2.26} we have
 	\begin{align}
 		\label{eq4.22}
 		E_{ij}=b_{ij},F_{ij}=0, F_j^i=0
 	\end{align}
 	Using equation ~\ref{eq4.22} into equation ~\ref{eq2.27}, we get
 	\begin{align}
 		\nonumber
 		D_{jk}^i=b_{jk}B^i+b_{0k}B_j^i+b_{0j}B_k^i-b_{0m}g^{im}B_{jk}\\ 
 		\label{eq4.23}
 		-A_k^mC_{jm}^i-A_j^mC_{km}^i+A_s^mC_{jkm}g^{is}\\ \nonumber
 		+\lambda^s(C_{sk}^mC_{jm}^i+C_{sj}^mC_{km}^i-C_{ms}^iC_{jk}^m)
 	\end{align}
 	
 	Using the equations ~\ref{eq4.2}, ~\ref{eq4.4}, ~\ref{eq4.5} and ~\ref{eq4.6} into equations ~\ref{eq2.28} to ~\ref{eq2.23}, we get
 	
 	\begin{align}
 		\label{eq4.24}
 		B_k&=(n+1)(2n+1)b_k+\frac{n+1}{\alpha}y_k,
 		B^i=bb^i+by^i\\
 		\label{eq4.25}
 		B_{ij}&=\frac{(n+1)\left\{a_{ij}\alpha^2-y_iy_j+2n(n+1)b_ib_j\alpha\right\}}{2\alpha^3}\\
 		\label{eq4.26}
 		B_j^i&=0\\
 		\label{eq4.27}
 		A_k^m&=0, \lambda^m=B^mb_{00}.
 	\end{align}
 	Using tensor contraction operation with equations ~\ref{eq4.25} and ~\ref{eq4.26} by $y^j$, we get $B_{i0}=0,$ $B_0^i=0.$ Further contracting equation ~\ref{eq4.27} by $y^k$ and using the fact that $B_0^i=0$, we get $A_0^m=B^mb_{00}.$
 	Contracting equation ~\ref{eq4.23} by $y^k$ and using the facts $B_{i0}=0$, $B_0^i=0$, $A_0^m=B^mb_{00}$ and   $C_{s0}^m=0$, $C_{0m}^i=0$, $C_{j0}^m=0$ obtained by contracting equations ~\ref{eq4.25}, ~\ref{eq4.26}, ~\ref{eq4.27} and ~\ref{eq3.6},  we get
 	\begin{align}
 		\label{eq4.28}
 		D_{j0}^i&=B^ib_{j0}+B_j^ib_{00}-b_{00}B^mC_{jmi}^i\\
 		\label{eq4.29}
 		D_{00}^i&=bb^ib_{00}+by^ib_{00}
 	\end{align}
 	Multiplying equation ~\ref{eq4.25} by $b_i$ and then using equations ~\ref{eq4.2}, ~\ref{eq4.21} and ~\ref{eq4.23},  we get
 	\begin{align}
 		\label{eq4.30}
 		b_iD_{j0}^i=bb_{j0}+bb_jb_{00}-bb_ib^mC_{jm}^ib_{00}
 	\end{align}
 	Now multiplying equation ~\ref{eq4.26} by $b_i$ and then using equation ~\ref{eq4.2} we get
 	\begin{align}
 		\label{eq4.31}
 		b_iD_{00}^i=\frac{b^2}{1+n(n+1)}b_{00 }
 	\end{align}
 	From equations ~\ref{eq4.14} and ~\ref{eq4.16} it is clear that
 	\begin{align}
 		\label{eq4.32}
 		b^mb_iC_{jm}^iB_\alpha^j=\frac{b^2}{1+n(n+1)}M_\alpha=0
 	\end{align}
 	Contracting the expression $b_{i|j}=b_{ij}-b_rD_{ij}^r$ by $y^i$ and $y^j$ respectively and then using equation ~\ref{eq4.31} we get
 	\begin{align*}
 		b_{i|j}y^iy^j=b_{00}-b_rD_{00}^r=\frac{b^2}{1+n(n+1)}b_{00}
 	\end{align*}
 	
 	Put $b_{i|j}=b_{ij}-b_rD_{ij}^r$ in equations   ~\ref{eq4.17} and  ~\ref{eq4.18} and then using equations ~\ref{eq4.27},  ~\ref{eq4.1} and  ~\ref{eq4.29} and the value of $b_{i|j}y^iy^j$ above, equations  ~\ref{eq4.20} and  ~\ref{eq4.21} can be written as
 	\begin{align}
 		\label{eq4.33}
 		\sqrt{\frac{b^2}{1+n(n+1)}}b_{i0}B_\alpha^i+bH_\alpha=0\\
 		\label{eq4.34}
 		\sqrt{\frac{b^2}{1+n(n+1)}}b_{00}+bH_0=0
 	\end{align}
 	From the equation ~\ref{eq4.31} it is clear that the condition $H_0=0$ is equivalent to $b_{00}=0,$ where $b_{ij}$ is independent of $y^i$. Since $y^i$ satisfy equation ~\ref{eq4.2}, the condition can be written as $b_{ij}y^iy^j=(b_iy^i)(c_jy^j)$ for some $c_j(x)$, so that we have
 	\begin{align}
 		\label{eq4.35}
 		2b_{ij}=b_ic_j+b_jc_i
 	\end{align} 
 Thus we shown that a Finslerian hypersurface $M^{n-1}$ will be  hyperplane of first kind if and only if $2b_{ij}=b_ic_j+b_jc_i$.\\

 Now we try to show that  second fundamental tensor $H_{\alpha\beta}$ of the hypersurface $M^{n-1}$ is proportional to its angular metric tensor $h_{\alpha\beta}$.
 
 For that, contracting equation  ~\ref{eq4.35} and using the fact that $b_iy^j=0$, we get $b_{00}=0$. This implies that the condition  $b_{00}=0$ and $2b_{ij}=b_ic_j+b_jc_i$ are equivalent. \\

 Multiplying equation ~\ref{eq4.35} by $B_\alpha^i$ and then $B_\beta^j$ and using equations ~\eqref{eq4.1} and ~\eqref{eq4.2}, we have
 \begin{align*}
 	2b_{ij}\times B_\alpha^i&=\left(b_iB_\alpha^i\right)c_j+b_jc_iB_\alpha^i\\
 	2b_{ij}B_\alpha^i&=0\times c_j+b_jc_iB_\alpha^i\\
 	2b_{ij}B_\alpha^i\times B_\beta^j&=\left(b_jB_\beta^j\right)c_iB_\alpha^i\\
 	2b_{ij}B_\alpha^iB_\beta^j&=0\times c_iB_\alpha^i\\
 	2b_{ij}B_\alpha^iB_\beta^j&=0\\
 	b_{ij}B_\alpha^iB_\beta^j&=0
 \end{align*}

 Again,  multiplying equation ~\ref{eq4.35} by $B_\alpha^i$ and $y^j$ and then using Equation   ~\eqref{eq4.1}, we have  
 
 \begin{align*}
 	2b_{ij}\times B_\alpha^iy^j&=\left(b_iB_\alpha^i\right)c_jy^j+\left(b_jy^j\right)c_iB_\alpha^i\\
 	2b_{ij}B_\alpha^iy^j&=0\times c_jy^j+0\times c_iB_\alpha^i\\
 	2b_{ij}B_\alpha^iy^j&=0\\
 	b_{ij}B_\alpha^iy^j&=0\\
 	b_{i0}B_\alpha^i&=0                   \quad\text{(contraction by $y^j$ is taking place)}
 \end{align*}
 
 Again, consider equation ~\eqref{eq4.35}
 \begin{align*}
 	2b_{ij}&=b_ic_j+b_jc_i\\
 	2b_{ij} y^j&=b_i\left(c_j y^j\right)+\left(b_j y^j\right)c_i      \quad\text{(multiplying  by $y^j$ both sides)}\\
 	2b_{i0}&=b_ic_0+0\times c_i                                               \quad\text{(contraction by $y^j$ is taking place)}\\
 	2b_{i0}&=b_ic_0\\
 	2b_{i0}b^i&=\left(b_ib^i\right)c_0                                         \quad\text{(multiplying  by $b^i$ both sides)}\\
 	2b_{i0}b^i&=b^2c_0                                                            \quad\text{ ($ \because b^2=b_ib^i$)}\\
 	b_{i0}b^i&=\frac{b^2c_0}{2}
 \end{align*}

  Using this in equation ~\ref{eq4.30}  gives $H_\alpha=0.$  Now using ~\ref{eq4.23} and ~\ref{eq4.24} and using $b_{00}=0$ and $b_{ij}B_\alpha^iB_\beta^j=0$, we get $\lambda^m=0,$ $A_j^iB_\beta^j=0$ and $B_{ij}B_\alpha^iB_\beta^j=\frac{1}{2\alpha}h_{\alpha\beta}$. Thus using the equations ~\ref{eq4.6}, ~\ref{eq4.7}, ~\ref{eq4.8}, ~\ref{eq4.12} and ~\ref{eq4.20}, we get\\
 \begin{align}
 	\label{eq4.36}
 	b_rD_{ij}^rB_\alpha^iB_\beta^j=-\frac{c_0b^2}{1+n(n+1)}h_{\alpha\beta}
 \end{align}
 Thus using the relation $b_{i|j}=b_{ij}-b_rD_{ij}^r$ and equation ~\ref{eq4.36}, equation ~\ref{eq4.19} reduces to
 \begin{align}
 	\label{eq4.37}
 	-\frac{c_0b^2}{1+n(n+1)}h_{\alpha\beta}+\sqrt{\frac{b^2}{1+n(n+1)}}H_{\alpha\beta}=0
 	\end{align} 
 
 \begin{align*}
 	H_{\alpha\beta}=\frac{c_0b}{\sqrt{1+n(n+1)}}h_{\alpha\beta}
 \end{align*} 

Thus we shown  that the second fundamental tensor $H_{\alpha\beta}$ of $M^{n-1}$ is proportional to its angular metric tensor $h_{\alpha\beta}$.

 \end{proof}
\begin{theorem}
	\label{theorem4.5}
Let $(M,F)$ be a Finsler manifold, where $F(\alpha,\beta)=\frac{(\alpha+\beta)^{n+1}}{\alpha^n}$, $n\in N$, is a generalized square metric and  $M^{n-1}$ be its associated  hypersurface.Then the hypersurface $M^{n-1}$ will be  hyperplane of second kind if and only if
\begin{align*}
	b_{ij}&=eb_ib_j
\end{align*}
\end{theorem}

\begin{proof}
	We know from    Lemma ~\ref{lemma3.8}, hypersurface  $M^{n-1}$ is a hyperplane of second kind  if $H_\alpha=0$ and $H_{\alpha\beta}=0$. Now we consider these two sufficient conditions one by one.
	\begin{enumerate}
		\item If  $H_{\alpha\beta}=0$, then equation ~\eqref{eq4.37} becomes
		\begin{align*}
			-\frac{c_0b^2}{1+n(n+1)}h_{\alpha\beta}+\sqrt{\frac{b^2}{1+n(n+1)}}\times 0=0\\
			\implies-\frac{c_0b^2}{1+n(n+1)}h_{\alpha\beta}=0\\
			\implies c_0=0\\
			\implies c_0=c_iy^i=0
		\end{align*}
		$\implies$ there exist a function $e(x)$ such that $c_i(x)=e(x)b_i(x)$ and this  $c_i(x)=e(x)b_i(x)$ forces $c_0$ to  vanish as follows:
		\begin{align*}
			c_0&=c_i(x)y^i\\
			&=e(x)b_i(x)y^i\\
			&=e(x)\left(b_i(x)y^i\right)\\
			&=e(x)\times 0                                        \quad\text{(as along the hypersurface $b_iy^i=0$)}\\
			\therefore c_0&=0
		\end{align*} 
		\item Again,  if  $H_\alpha=0$, then  Lemma ~\ref{lemma3.6} and Lemma ~\ref{lemma3.7} imply $H_0=0$.\\
		We have already shown above $H_0=0$ is equivalent to 
		\begin{align*}
			2b_{ij}=b_ic_j+b_jc_i
		\end{align*}
	\end{enumerate}

	Now we combine case 1 and case 2. For that, put the value of  $c_i(x)=e(x)b_i(x)$  obtained in case 1 to equation obtained in case 2,  we get
	\begin{align}
		\nonumber
		2b_{ij}&=b_ie(x)b_j(x)+b_je(x)b_i(x)\\
		\nonumber&=e\left(b_ib_j+b_jb_i\right)\\
		\nonumber&=e \times 2b_ib_j\\
		\nonumber\therefore 2b_{ij}&=2eb_ib_j\\
		\label{eq4.38} \implies b_{ij}&=eb_ib_j
	\end{align}

Thus we shown that the 	hypersurface $M^{n-1}$ of the Finsler manifold $(M,F)$  will be  hyperplane of second kind iff $b_{ij}=eb_ib_j$.

\end{proof}

\begin{theorem}
	\label{theorem4.6}
Let $(M,F)$ be a Finsler manifold, where $F(\alpha,\beta)=\frac{(\alpha+\beta)^{n+1}}{\alpha^n}$, $n\in N$, is a generalized square metric and  $M^{n-1}$ be its associated  hypersurface. Then the hypersurface $M^{n-1}$ will not be  hyperplane of third kind.
\end{theorem}

\begin{proof}
We know  from sufficient conditions of  Lemma ~\ref{lemma3.9} a hypersurface becomes a hyperplane of third kind  if $H_\alpha=0,$ $H_{\alpha\beta}=0$ and $M_{\alpha\beta}=0$. Now we consider these three sufficient conditions one by one.
\begin{enumerate}
\item If $H_\alpha=0$, then we get the condition $2b_{ij}=b_ie(x)b_j(x)+b_je(x)b_i(x)$, which has already been proved above and is termed as the condition of hyperplane of first kind.
\item If $H_{\alpha\beta}=0$, then we get the condition $b_{ij}=eb_ib_j$, which has already been proved above and is termed as the condition of hyperplane of second kind.
\item Now put  $M_{\alpha\beta}=0$ in Equation ~\eqref{eq4.15}, we get
\begin{align*}
	0=\frac{(n+1)}{2\alpha}\left(\sqrt{\frac{b^2}{1+n(n+1)}}\right)h_{\alpha\beta}
\end{align*} 
which implies that  no condition could be deduced to satisfy  $M_{\alpha\beta}=0$, i.e., it is impossible to find a condition under which a hypersurface becomes a hyperplane of third kind, as term on the R.H.S.  of the above equation can never be zero.
\end{enumerate}
Finally, we shown that hypersurface  $M^{n-1}$ is not a hyperplane of third kind.
\end{proof}

\begin{corollary}
	Let $(M,F)$ be a Finsler manifold, where $F$ may be any of the following Finsler  metrics obtained by generalized square metric $F=\frac{(\alpha+\beta)^{n+1}}{\alpha^n}$, $n=1,2,3,......$
	\begin{enumerate}
		\item $F=\frac{(\alpha+\beta)^2}{\alpha}$                 \quad\text{(popularly known as square metric)}
		\item $F=\frac{(\alpha+\beta)^3}{\alpha^2}$
		\item $F=\frac{(\alpha+\beta)^4}{\alpha^3}$, etc., \\
		Also let $M^{n-1}$ be the corresponding hypersurfaces of the given Finsler manifold $(M,F)$. Then, in either case, show that the hypersurface is a hyperplane of first kind, second kind and not of the third kind.
	\end{enumerate} 
\end{corollary}
\begin{proof}
By Theorems ~\ref{theorem4.4}, ~\ref{theorem4.5} and  ~\ref{theorem4.6} it can be easily deduce that the hypersurfaces  $M^{n-1}$  corresponding  to different  Finsler manifolds $(M,F)$ are a hyperplane of first kind, second kind and not of the third kind. It is remarkable that corresponding author  has already been published a paper ~\cite{rafe2018hypersurface} on part (1) of this corollary.
	\end{proof}
	\begin{conclusion}
	Now after all, one may ask why authors of the article is so interested to carry forward the theory of hypersurface over Finsler space with generalized square metric $F(x,y)=\frac{[\alpha(x,y)+\beta(x,y)]^{n+1}}{[\alpha(x,y)]^n}$, where $n=1,2,3,...$. My answer is very assertive that generalization of any theory is always fascinating due to its nature to bring various special cases under one umbrella. For example, the theorems that  we have proved works for every natural number $n\in N$.
	\end{conclusion}

\printbibliography[
heading=bibintoc,
title={References}
] 

\end{document}